\documentclass[12pt]{article}

\usepackage{graphicx}
\usepackage{enumerate}
\usepackage{amssymb}
\usepackage{amsmath}
\usepackage{mathrsfs}
\usepackage{latexsym}
\usepackage{psfrag}
\usepackage{mathrsfs}
\usepackage{verbatim}

\newcommand{\qed}{\hfill $\square$\\
\vspace{0.1cm}}

\newtheorem{theorem}{Theorem}[section]
\newtheorem{definition}[theorem]{Definition}
\newtheorem{lemma}[theorem]{Lemma}
\newtheorem{fact}[theorem]{Fact}
\newtheorem{proposition}[theorem]{Proposition}

\newtheorem{conjecture}{Conjecture}[section]
\newtheorem{question}[theorem]{Question}
\newenvironment{proof}{\noindent{\em Proof.}}{\qed}

\newcommand{\B}{\mathcal{B}}
\newcommand{\C}{\mathcal{C}}
\newcommand{\F}{\mathcal{F}}
\newcommand{\G}{\mathcal{G}}

\begin{document}

\title{The Boolean Rainbow Ramsey Number of Antichains, Boolean Posets, and Chains}

\author{
Hong-Bin Chen
\thanks{Department of Applied Mathematics, National Chung Hsing University, Taichung 40227, Taiwan
{\tt Email:andanchen@gmail.com} supported by MOST 107-2115-M-035 -003 -MY2.}
\and
Yen-Jen Cheng
\thanks{Department of Mathematics, National Taiwan Normal University, Taipei 116, Taiwan
{\tt Email:yjc7755@gmail.com}}
\and
Wei-Tian Li
\thanks{Department of Applied Mathematics, National Chung Hsing University, Taichung 40227, Taiwan
{\tt Email:weitianli@nchu.edu.tw} supported by MOST-107-2115-M-005 -002 -MY2.}
\and
Chia-An Liu
\thanks{Department of Mathematics, Xiamen University Malaysia, Selangor 43900, Malaysia.
{\tt Email:chiaan.liu@xmu.edu.my}}
}

\date{\small \today}

\maketitle

\begin{abstract}
Motivated by the paper of Axenovich and Walzer~\cite{AW},
we study the Ramsey-type problems on the Boolean lattices.
Given posets $P$ and $Q$, we look for the smallest Boolean lattice $\B_N$ such that
any coloring on elements of $\B_N$ must contain a monochromatic $P$ or a rainbow $Q$.
This number $N$ is called the Boolean rainbow Ramsey number of $P$ and $Q$ in the paper.

Particularly, we determine the exact values of the Boolean rainbow Ramsey number for $P$ and $Q$ being the antichains, the Boolean posets, or the chains.
From these results, we also give some general upper and lower bounds of
the Boolean rainbow Ramsey number for general $P$ and $Q$ in terms of the poset parameters.
\end{abstract}

{\noindent} Keywords: Boolean lattices, posets, Ramsey theory, rainbow colorings.

\section{Introduction}

A {\em poset} $P=(P,\le_P)$ is a set $P$ equipped with a partial order $\le_P$.
In this paper, we study the Ramsey-type problems of the well-known poset, {\em Boolean lattices},
$\mathcal{B}_n$ whose underlying set is the collection of all subsets of $[n]:=\{1,2,\ldots, n\}$ and the partial order is the inclusion relation on sets.
The {\em $k$-th level} of $\B_n$ is the collection of all $k$-subsets of $[n]$, denoted as $\binom{[n]}{k}$.
A family $\F$ of subsets is isomorphic to a poset $P$ if there exists an order-preserving bijection $\phi$ between $P$ and $\F$, i.e. $P\stackrel{\phi}{\longleftrightarrow}\F$ such that for any $x_1, x_2\in P$, $x_1 <_P x_2$ if and only if
$\phi(x_1 )\subset \phi(x_2)$.
If a family $\F$ contains a subfamily $\G$ isomorphic to $P$, we will say $\F$ contains $P$ as a {\em strong subposet} ($\F$ contains $P$, for short),
or say the subsets in $\G$ form a copy of $P$.
A {\em coloring} ({\em $k$-coloring}) on $\B_n $ is a mapping $c$ from $\B_n $ to a set of positive integers (to $[k]$).
Given a coloring $c$ on $\B_n$, we say $\B_n$ contains a {\em monochromatic} $P$ under $c$ if there is a family of subsets of the same color containing $P$.

In the literature of Ramsey theory, the Ramsey problems have been greatly studied on the set systems (hypergraphs),
the graphs, the planes, and the general posets, for example see~\cite{AGLM,CFS1,CFS2,FMO,KT,T99}.
The following type of Ramsey problems on Boolean lattices has been recently studied by Axenovich and Walzer~\cite{AW}:

\begin{question}
Given posets $P_1,\ldots,P_k$, find the least integer $n$ such that
for any $k$-coloring on $\mathcal{B}_n$, it always contains a monochromatic $P_i$ of color $i$ for some $i$?
\end{question}

In their paper, they called such a number the poset Ramsey number.
Since there are many Ramsey properties investigated on the families of posets with various parameters,
to make it more precise, we suggest the name {\em the Boolean Ramsey number} and use it in this paper.
Let us define the term formally below:

\begin{definition}[The Boolean Ramsey number]
Given posets $P_1,\ldots,P_k$, the Boolean Ramsey number
 $R(P_1,\ldots,P_k)$ is the minimum integer $n$ such that
for any $k$-coloring on $\mathcal{B}_n$, it always contains a monochromatic $P_i$ of color $i$ for some $i$.
Moreover, if $P_i=P$ for all $1\le i\le k$, then we use $R_k(P)$ to denote $R(P_1,\ldots,P_k)$.
\end{definition}

The number $R_k(P)$ is shown to be finite in~\cite{AW}:
\begin{theorem}{\rm[Theorem 6,~\cite{AW}]}\label{AW}
For any poset that is not an antichain, $R_k(P)=\Theta(k)$.
\end{theorem}

Given a coloring $c$ on $\B_n$,
we say $\B_n$ contains a {\em rainbow} $P$ if it contains $P$ and all the subsets forming $P$ are of distinct colors under $c$.
Imitating the rainbow Ramsey number in graph theory, we give the following definition:

\begin{definition}[The Boolean rainbow Ramsey number]
Given two collections of posets $\mathcal{P}=\{P_1,\ldots,P_r\}$ and $\mathcal{Q}=\{Q_1,\ldots,Q_s\}$,
the Boolean rainbow Ramsey number $RR(\mathcal{P},\mathcal{Q})$ is the minimum integer $n$ such that
for any coloring $c$ on $\B_n$, it always contains some monochromatic $P_i$
or some rainbow $Q_j$.
If $\mathcal{P}=\{P\}$ (resp. $\mathcal{Q}=\{Q\}$), then we will simply use $P$ (resp. $Q$) instead of $\{P\}$ (resp. $\{Q\}$).
\end{definition}

It is not evident that $RR(P,Q)$ is finite for any $P$ and $Q$.
Nevertheless, Johnston, Lu, and Milans~\cite{JLM} proved that for $n$ sufficiently large, any coloring on $\B_n$ contains either a monochromatic {\em Boolean algebra}  $B_{alg(r)}$ or a rainbow $B_{alg(s)}$ for any given positive integers $r$ and $s$. A {\em Boolean algebra} $B_{alg(r)}$ in $\B_n$
is such a family $\{S_0\cup(\cup_{i\in I} S_{i}) \mid I\subset [r] \}$, where $S_0,S_1,\ldots, S_r$ are mutually disjoint subsets of $[n]$.

\begin{theorem}{\rm[Theorem 5,~\cite{JLM}]}\label{JLM}
For $N\ge r2^{(2r+1)2^{s-1}-2}$, $\B_N$ contains either a rainbow  $B_{alg(r)}$ or a monochromatic $B_{alg(s)}$ under any coloring.
\end{theorem}

Containing the Boolean algebra $B_{alg(r)}$ implies containing $\B_r$
since  $B_{alg(r)}$ is isomorphic to $\B_r$.
So $RR(P,Q)$ is finite when $P$ and $Q$ are both Boolean lattices.
In addition, Trotter~\cite{T75} introduced the {\em 2-dimension} of a poset $P$,
$\dim_2(P)$ is the minimum number $n$ for which $\B_n$ contains $P$ and proved that $\dim_2(P)$ is finite for any poset $P$.
As a consequence, $RR(P,Q)$ exists for any $P$ and $Q$.

Analogous to~\cite{AW}, Cox and Stolee~\cite{CS} studied the existence of the monochromatic weak subposets in the Boolean lattices.
Here a {\em weak subposet} means an injection from the poset to the Boolean lattice,
which preserves the inclusion relation but not necessary the non-inclusion relation.
A very recent paper~\cite{C7} by the third author and others presents the results on
the Ramsey properties of both types of subposets in the Boolean lattices.

In this paper, we study the strong version of the Boolean rainbow Ramsey number and the relations between
it and the Boolean Ramsey number.
We determine the exact value of $RR(P,Q)$ for specific posets and use the results to derive the upper and lower bounds for $RR(P,Q)$ for general $P$ and $Q$.
We focus on the {\em antichains} $A_n$, the {\em Boolean posets} $B_n$, and the {\em chains} $C_n$,
where $A_n$ is a set of $n$ elements without partial order relation,
$B_n$ is isomorphic to $\B_n$, and $C_n$ is a totally ordered set of $n$ elements.
In the paper, the scribe $\B_n$ refers to the underlying Boolean lattice we color,
and the capital $B_n$ refers to the desired monochromatic or rainbow Boolean posets.
Table~\ref{TAB} is the summary of our results on $RR(P,Q)$ when $P$ and $Q$ are one of the three type of posets.
\begin{table}[h]
\begin{center}
\begin{tabular}{|c|c|c|}
\hline $P$ & $Q$ & $RR(P,Q)$\\
\hline $A_m$ & $A_n $ & $\min\{N\mid \binom{N}{\lfloor N/2\rfloor}\ge (m-1)(n-1)+1\}$. \\
\hline $A_2$ & $C_n$ & $n$ \\
\hline $B_m$  & $B_1$ & $m$\\
\hline $B_1$ & $B_n$ & $2^n-1$\\
\hline $B_2$ & $B_2$ & 6 \\
\hline $C_m$ & $A_n$ & $\left\{
\begin{array}{ll}
n+2,& m=2\mbox{ and }n\ge 3 \\
(m-1)(n-1)+2,& m=n=2,\mbox{ or }m\ge 3\mbox{ and } n\ge 2\\

\end{array}
\right.$ \\
\hline $C_m$ & $B_n$ & $(m-1)(2^n-1)$\\
\hline $C_m$ & $C_n$ & $(m-1)(n-1)$ \\
\hline
\end{tabular}
\end{center}
\caption{$RR(P,Q)$ of antichains, Boolean posets, and chains}\label{TAB}
\end{table}

The rest of the paper is organized as follows.
In Section 2, we determine all but one values listed in Table~\ref{TAB}.
The exception $RR(B_2,B_2)$ will be postponed to Section 3 after we establish some bounds between $RR(P,Q)$ and $R_k(P)$.
The upper and lower bounds of $RR(P,Q)$ for general $P$ and $Q$ will be studied
and a by-product of the Boolean Ramsey number $R_3(B_2)=6$ will be shown in Section 3.
The last section contains the discussions of the upper bounds for $R_k(P)$ and $RR(P,Q)$ and the Forbidden subposet problems.

\section{The exact Boolean rainbow Ramsey numbers}

When $P$ and $Q$ are both chains or both antichains, the Boolean rainbow Ramsey numbers can be determined by simple arguments.

\begin{proposition}\label{CMCN}
For the chains, we have
 $RR(C_m,C_n)=(m-1)(n-1)$.
\end{proposition}

\begin{proof}
For $N<(m-1)(n-1)$, we give a coloring $c$ to $\B_N$ by coloring $i$ to the subsets in the consecutive $m-1$ levels
$\binom{[N]}{(i-1)(m-1)}\cup\binom{[N]}{(i-1)(m-1)+1}\cdots\cup\binom{[N]}{i(m-1)-1}$.
Since a chain $C_m$ in the Boolean lattices consists of $m$ subsets of distinct sizes, $\B_N$ does not
contain a monochromatic $C_m$ under $c$.
In the coloring $c$, the number of color classes is at most $n-1$, so $\B_N$ does not contain a rainbow $C_n$ as well.
When $N\ge (m-1)(n-1)$, consider any coloring on the chain
$\varnothing \subset [1] \subset [2]\subset \cdots\subset [N]$ in $\B_N$.
By pigeonhole principle, it contains at least $n$ subsets of distinct colors or at least $m$ subsets of the same color, as desired.
\end{proof}

\begin{proposition}\label{AMAN}
For the antichains, we have
 $RR(A_m,A_n)=N_{m,n}$, where $N_{m,n}$ is the minimum integer such that $\binom{N_{m,n}}{\lfloor N_{m,n}/2\rfloor}\ge (m-1)(n-1)+1$.
\end{proposition}

\begin{proof}
For $N<N_{m,n}$, partition $\B_N$ into $\binom{N}{\lfloor N/2\rfloor}$ chains, $\C_1, \ldots$, $\C_{\binom{N}{\lfloor N/2\rfloor}}$
using the symmetric chain decompositions~\cite{BEK,GK}. Then color every $m-1$ chains $\C_{(i-1)(m-1)+1},\ldots$, $\C_{i(m-1)}$ by $i$.
Hence any antichain in a color class has size at most $m-1$. Also, the coloring contains at most $n-1$ colors. Thus, there is no rainbow $A_n$.
For $N\ge N_{m,n}$, the level $\binom{[N]}{\lfloor N/2\rfloor}$ contains at least  $(m-1)(n-1)+1$ subsets.
Then any coloring on $\B_N$ contains at least $m$ subsets of the same color or $n$ subsets of distinct colors in the level $\binom{[N]}{\lfloor N/2\rfloor}$.
\end{proof}

Determining $RR(B_m,B_n)$ is not as simple as the previous tasks.
We manage to solve $RR(B_m,B_n)$ for certain cases of $m$ and $n$ in Theorem~\ref{B2B2}.
The proof of Theorem~\ref{B2B2} (3) is involved in more connections between the Boolean Ramsey numbers and the Boolean rainbow Ramsey numbers.
Here we will first prove (1) and (2)
postpone the proof of (3) to next section.

\begin{theorem}\label{B2B2}
For the Boolean posets, we have
\begin{description}
\item{(1)} $RR(B_m,B_1)=m$;
\item{(2)} $RR(B_1,B_n)=2^n-1$;
\item{(3)} $RR(B_2,B_2)=6$.
\end{description}
  \end{theorem}

\begin{proof}
(1) For $N<m$, we color every set by the same color to avoid a monochromatic $B_m$ and a rainbow $B_1$ in $\mathcal{B}_N$.
When $N\ge m$, let $c$ be any coloring  on $\mathcal{B}_N$.
If no rainbow $B_1$ exists under $c$,
then $c(X)=c(\varnothing)$ for any $X\subseteq [N]$.
This leads to a monochromatic $B_m$ in $\mathcal{B}_N$.

(2) For $N< 2^n-1$, we can color all subsets in the level $\binom{[N]}{i}$ by  $i$.
Then every color class is an antichain.
Moreover, to have a rainbow $B_n$, we need at least $2^n$ different colors,
which is greater than the number of color in the coloring.
Thus, $RR(B_1,B_n)\ge 2^n-1$.

For $N\ge 2^n-1$, let $S$ be the union of two disjoint sets $\{x_1,\ldots, x_n\}$ and $[N-n]$,
and $\B_{S}:=\{\{X\mid X\subseteq S\},\subseteq\}$.
Without loss of generality, we color on $\B_{S}$.
Assume $\B_S$ does not contain a monochromatic $B_1$ under a coloring $c$.
Then we will prove that it contains a rainbow $B_n$ under $c$.

For the subsets of $\{x_1,\ldots, x_n\}$, we fix an arbitrary order of subsets of the same size.
Let $X_{h,k}\subset \{x_1,\ldots, x_n\}$ be the $k$-th subset of size $h$.
First define $S_{0,1}=X_{0,1}$ and $m_{0,1}=0$.  For $h\ge 1$, let $m_{h,k}$ be the smallest nonnegative integer
such that $m_{h,k}\ge m_{h-1,k'}$ for any $X_{h-1,k'}\subset X_{h,k}$ and
$c(X_{h,k}\cup[m_{h,k}])$ ($X_{h,k}\cup[m_{h,k}]=X_{h,k}$ if $m_{h,k}=0$) is distinct from any color in $\{c(S_{h',k'})\mid h'<h,\mbox{ or } h'=h\mbox{ and }k'<k\}$.
If $m_{h,k}\le N-n$, then define $S_{h,k}=X_{h,k}\cup[m_{h,k}]$.
Clearly, if $S_{h,k}$ exists for any $0\le h\le n$ and $1\le k \le \binom{n}{h}$, then these $2^n$ sets form a copy of rainbow $B_n$.
Thus, it suffices to show that when $N\ge 2^n-1$, all $S_{h,k}$ exist.

Suppose, on the contrary, some $S_{h,k}$ does not exist.
Let $h_1$ and $k_1$ be the ``smallest'' pair such that $S_{h_1,k_1}$ does not exist but $S_{h,k}$ exists for $h<h_1$ and $1\le k\le \binom{n}{h}$, or for $h=h_1$ and $k<k_1$.
For $h<h_1$ and $1\le k\le \binom{n}{h}$, or for $h=h_1$ and $k<k_1$, we call a chain $\C_{h,k}$ in $\B_s$ a {\em principal chain} of $S_{h,k}$ if it consists of $|S_{h,k}|+1$ sets with the largest one $S_{h,k}$ such that all the colors of the sets in $\C_{h,k}$ are in $\{c(S_{h',k'})\mid h'<h,\mbox{ or } h'=h\mbox{ and }k'\le k\}$.
We claim that the principal chain exists for $h<h_1$ and $1\le k\le \binom{n}{h}$, or for $h=h_1$ and $k<k_1$.
This can be done by induction on $h$. It is clear $\C_{0,1}=\{\varnothing\}$.
For $S_{h,k}$ with $h\ge 1$, let $m_{h-1,k^*}=\max_{k'}\{m_{h-1,k'}\mid X_{h-1,k'}\subset X_{h,k}\}$.
By induction, the principal chain $\C_{h-1,k^*}$ exists, and all the colors of the sets in $\C_{h-1,k^*}$ are in
$\{c(S_{h',k'})\mid h'<h-1,\mbox{ or } h'=h-1\mbox{ and }k'\le k^*\}$.
Moreover, by the definition of $m_{h,k}$, the colors $c(X_{h,k}\cup [i])$, $m_{h-1,k^*}\le i \le m_{h,k}-1$,
are in $\{c(S_{h',k'})\mid h'<h,\mbox{ or } h'=h\mbox{ and }k'< k\}$. Concatenating the chain $\C_{h-1,k^*}$ to
the chain $\{X_{h,k}\cup[i]\mid m_{h-1,k^*}\le i\le  m_{h,k}\}$, we obtain the principal chain of $S_{h,k}$.
Since $S_{h_1,k_1}$ does not exist, it means the colors
$c(X_{h_1,k_1}\cup[i])$, $m_{h_1-1,k^{**}}\le i \le N-n$, are all in
$\{c(S_{h,k})\mid h<h_1,\mbox{ or } h=h_1\mbox{ and }k< k_1\}$, where $m_{h_1-1,k^{**}}=\max\{m_{h_1-1,k}\mid X_{h_1-1,k}\subset X_{h_1,k_1}\}$.
In other words, the number of the colors of the subsets in the chain
$\C_{h_1-1,k^{**}}\cup\{X_{h_1,k_1}\cup [i]\mid  m_{h_1-1,k^{**}}\le i\le N -n\}$ is at most $\sum_{i=0}^{h_1-1}\binom{n}{i}+k_1-1$.
Note that this chain is a rainbow chain. So we have
\[|S_{h_1-1,k^{**}}|+1+(N-n-m_{h_1-1,k^{**}}+1)\le \sum_{i=0}^{h_1-1}\binom{n}{i}+k_1-1.\]
Since $|S_{h_1-1,k^{**}}|=|X_{h_1-1,k^{**}}|+m_{h_1-1,k^{**}}$, we simplify the previous inequality to get
$h_1+N-n+1\le \sum_{i=0}^{h_1-1}\binom{n}{i}+k_1-1$, and hence
$N\le \sum_{i=0}^{h_1-1}\binom{n}{i}+k_1-1 -h_1+n-1<2^n-1$. This contradicts to our assumption of $N$.
\end{proof}

The idea of the principal chain in the proof of Theorem~\ref{B2B2} (2) can be applied to solve the Boolean rainbow Ramsey number for mixed types of posets:

\begin{theorem}\label{CMBN}
For the chains and Boolean posets, we have
$RR(C_m,B_n)=(m-1)(2^n-1)$.
\end{theorem}

\begin{proof} For $N<(m-1)(2^n-1)$, we give a coloring to $\B_N$ by coloring $i$ to the sets in the consecutive $m-1$ levels
$\binom{[N]}{(i-1)(m-1)}\cup\binom{[N]}{(i-1)(m-1)+1}\cdots\cup\binom{[N]}{i(m-1)-1}$. On the one hand,
since a color class contains at most $m-1$ different sizes of subsets,
$\B_N$ does not contain a monochromatic $C_m$.
On the other hand, the number of colors is at most $2^n-1$. Hence it does not contain a rainbow $B_n$.

Now for $N\ge (m-1)(2^n-1)$ define $\B_S$ as in the proof of Theorem~\ref{B2B2}(2).
We consider a coloring $c$ on $\B_s$ not containing a monochromatic $C_m$,
and let $X_{h,k}$, $m_{h,k}$, $S_{h,k}$, and $\C_{h,k}$ be defined similarly as before.
If some $S_{h_1,k_1}$ does not exist, then the colors
$c(X_{h_1,k_1}\cup[i])$, $m_{h_1-1,k^{**}}\le i \le N-n$, are all in
$\{c(S_{h,k})\mid h<h_1,\mbox{ or } h=h_1\mbox{ and }k< k_1\}$, where $m_{h_1-1,k^{**}}=\max\{m_{h_1-1,k}\mid X_{h_1-1,k}\subset X_{h_1,k_1}\}$.
In this chain, the number of subsets of the same color is at most $m-1$. Hence, we have
\[|S_{h_1-1,k^{**}}|+1+(N-n-m_{h_1-1,k^{**}}+1)\le (m-1)\left(\sum_{i=0}^{h_1-1}\binom{n}{i}+k_1-1\right).\]
Again, simplifying it leads to $N< (m-1)(2^n-1)$, a contradiction.
So $\B_S$ must contain a rainbow $B_n$ under $c$.
\end{proof}

Two families $\F$ and $\G$ of subsets are said to be {\em incomparable} if for any $F\in \F$ and $G\in \G$,
neither $F\subset G$ nor $G\subset F$.
The following structure in the Boolean lattices will help us to determine $RR(C_m,A_n)$.

\begin{lemma}\label{INC}
For $n\ge 2$, there exist $n$ incomparable chains $\C_1,\ldots, \C_n$ in $\B_N$
with
\[
|\C_i|=
\left\{
\begin{array}{ll}
i,& \mbox{ if }N=n+2;\\
(m-1)(i-1)+1 ,&\mbox{ if } N=(m-1)(n-1)+2 \mbox{ and }m=n=2\,(\mbox{or } m\ge 3).
\end{array}
\right.
\]
\end{lemma}

\begin{proof}
We prove both cases by induction on $n$ and constructing the families of chains recursively. The proofs are indeed similar.
When $N=n+2$, the base case is $\B_4$.
Take $\C_1=\{\{1,2\}\}$ and $\C_2=\{\{4\},\{1,4\}\}$ as the desired incomparable chains.
Suppose we already have $\C_1,\ldots, \C_n$ satisfying the conditions in $\B_{n+2}$.
Then define $\C_i'=\{[n+2]-X \mid X\in \C_i\}$, for $1\le i\le n$, and $\C_{n+1}'=\{\{n+3\}\cup [j]\mid 0\le j\le n\}$.
Clearly, for $1\le i<j\le n$, $\C'_i$ and $\C'_j$ are incomparable. For $W\in \C'_i$ and $Z\in C'_{n+1}$,
we have not only $n+3\in Z\setminus W$ but also $n+2\in  W\setminus Z$ if $i\le n-1$ and $n+1\in W\setminus Z$ if $i=n$.
Hence $\C_1'\ldots, \C_{n+1}'$ are incomparable.

When $N=(m-1)(n-1)+2$ and $n=2$, we take $\C_1=\{\{m\}\}$ and $\C_2=\{\{m+1\}\cup[j]\mid 0\le j\le m-1\}$.
For $m\ge 2$, we see that $\C_1$ and $\C_2$ are incomparable.
In particular, the case $m=n=2$ is verified.
For $m\ge 3$, suppose we already have $\C_1,\ldots, \C_n$ satisfying the conditions in $\B_{(m-1)(n-1)+2}$.
Then define $\C'_i=\{[(m-1)n+1]-X\mid X\in \C_i\}$, for $1\le i\le n$,
and $\C_{n+1}'=\{\{(m-1)n+2\}\cup [j]\mid 0\le j\le (m-1)n\}$.
As before, $\C'_i$ and $\C'_j$ are incomparable for $1\le i\le n$.
For any $W\in \C_i'$ with $i\le n$ and $Z\in \C_{n+1}'$,
note that $(m-1)n+1>(m-1)(n-1)+2$ if $m\ge 3$, so $(m-1)n+1\in W\setminus Z$.
Also, $(m-1)n+2\in Z\setminus W$. Therefore, $\C_1',\ldots, \C_{n+1}'$ are incomparable.
\end{proof}

\begin{theorem}\label{CMAN}
For the chains and antichains, we have
\[RR(C_m,A_n)=
\left\{
\begin{array}{ll}
n+2,& m=2\mbox{ and }n\ge 3 \\
(m-1)(n-1)+2,& m=n=2,\mbox{ or }m\ge 3\mbox{ and } n\ge 2\\

\end{array}
\right.
.\]
\end{theorem}

\begin{proof}
The constructions in Lemma~\ref{INC} imply the upper bounds $RR(C_2,A_n)\le n+2$ for $n\ge 3$, and $RR(C_m,A_n)\le (m-1)(n-1)+2$ for $m=n=2$, or for $m\ge 3$.
Let $\C_1,\ldots, \C_n$ be the chains constructed in Lemma~\ref{INC}.
In both cases, there exist at least $i$ colors on the subsets in $\C_i$,
otherwise it contains a monochromatic $C_m$.
However, if there are $i$ colors on $\C_i$,
then we can pick a set from each chain in order so that all sets are of distinct colors.
The $n$ sets will form a rainbow $A_n$ since all the chains are incomparable.

To establish the lower bound $RR(C_2,A_n)>n+1$ for $n\ge 3$, we assign color $i+1$ to all subsets in $\binom{[n+1]}{i}$ for $0\le i\le n+1$.
Clearly, there is no monochromatic $C_2$ under this coloring. If it contains rainbow $A_n$, then neither $\varnothing$ nor
$[n+1]$ can be in the rainbow antichain. So the rainbow $A_n$ must consists of one $X_i$ in each $\binom{[n+1]}{i}$ for $1\le i\le n$.
However, if $X_1\not\subset X_n$, then any other $X_i$ either contains $X_1$ or is contained in the complementary set of $X_1$, namely $X_n$.
As a consequence, the rainbow $A_n$ does not exist.

Next, we show $RR(C_m,A_n)>(m-1)(n-1)+1$ for $m=n=2$, or for $m\ge 3$ and $n\ge 2$.
Let $N=(m-1)(n-1)+1$. Assign color $1+\lceil \frac{|X|}{m-1}\rceil$ to each $X\subseteq [N]$.
This gives an $(n+1)$-coloring on $\B_N$ such that no color appears on $m$ distinct sizes of subsets.
So, there is no monochromatic $C_m$ under the coloring. Meanwhile, color 1 and $n+1$ appear only on
$\varnothing$ and $[N]$, respectively. Any family contains subsets of $n$ distinct colors must
contain at least one of the two sets, and thus cannot be an antichain $A_n$.
\end{proof}

Switching the antichain and chain in Theorem~\ref{CMAN} enhances the difficulty. We only solve the first step.

\begin{theorem}\label{A2CN} For $n\ge 2$,
$RR(A_2,C_n)=n$.
\end{theorem}

\begin{proof}
For $N<n$, we color the subsets $\varnothing$ and $[N]$ by the same color, and each of the remaining subsets by one color.
Clearly, there does not exist monochromatic $A_2$, and every longest rainbow chain has size $N\le n-1$.

To show $RR(A_2,C_n)\le n$, we begin with $n=2$.
For $\B_2$, if there is a coloring on $\B_2$ so that no monochromatic $A_2$ exists under the coloring,
then one of $\{1\}$ and $\{2\}$ together with $\varnothing$ form a rainbow $C_2$.
So $RR(A_2,C_2)\le 2$.
Assume there exists some $n$ for which $RR(A_2,C_n)\ge n+1$.
Let $n^*$ be the minimum $n\ge 3$ satisfying the inequality.
We will show such $n^*$ cannot exist.
Let $c$ be a coloring on $\B_{n^*}$ such that no monochromatic $A_2$ under $c$.
Suppose $c(\varnothing)\neq c([n^*])$. May assume $c(\varnothing)=1$ and $ c([n^*])=2$.
Note that the subsets of the same color must form a chain.
We may assume that all subsets of color 2 contain the element $n^*$.
By the minimality of $n^*$, the coloring $c$ restricted to $\B_{n^*-1}$ contains a rainbow $C_{n^*-1}$.
Since every subsets in $\B_{n^*-1}$ are not of color 2, the rainbow chain we just obtained together with $[n^*]$
form a rainbow chain $C_n$. This is a contradiction.
Else suppose $c(\varnothing)= c([n^*])=1$. As before, subsets of the same color must form a chain.
So we may assume if there are other subsets, in addition to $\varnothing$ and  $[n^*]$, are of color 1, then
all of them contain $n^*$.
Assume $c([n^*-1])=2$. We observe all subsets in $\B_{n^*-1}$, and replace the color of $\varnothing$ by $2$.
Under the new coloring, $\B_{n^*-1}$ contains a rainbow $C_{n^*-1}$ by the minimality of $n^*$.
This rainbow chain contains either $\varnothing$ or  $[n^*-1]$.
Indeed, if it contains $\varnothing$ but not  $[n^*-1]$, then we replace $\varnothing$ by $[n^*-1]$
and still have a rainbow $C_{n^*-1}$.
The chain together with $\varnothing$ form a rainbow $C_{n^*}$ under the original coloring $c$.
This also makes a contradiction. Hence  $n^*$ does not exist, and  $RR(A_2,C_n)\le n$ holds for all $n\ge 2$.
\end{proof}

\section{Bounds for $RR(P,Q)$}

At the beginning of this section, we give bounds for the Boolean rainbow Ramsey number of general posets.
A lower bound for $RR(P,Q)$ is $ R_{|Q|-1}(P)$. Since if $n=R_{|Q|-1}(P)-1$, there exists a $(|Q|-1)$-coloring on $
\B_n$ containing no monochromatic $P$ and also no rainbow $Q$ because of the insufficiency of colors.
Besides, some lower bounds can be obtained in terms of the parameters of the posets.
The {\em height} and {\em width} of a poset $P$, $h(P)$ and $w(P)$, are the sizes of a maximum chain and antichain of $P$, respectively.
The following lower bounds for $RR(P,Q)$ are the consequences of the results in the previous section:

\begin{fact}\label{LWB}
For any posets $P$ and $Q$,
we have
\begin{description}
\item{(1)} $RR(P,Q)\ge (h(P)-1)(h(Q)-1)$;
\item{(2)} $RR(P,Q)\ge N_{w(P),w(Q)}$;
\item{(3)} $RR(P,Q)\ge \left\{
\begin{array}{ll}
w(Q)+2,& h(P)=2\mbox{ and }w(Q)\ge 3 \\
(h(P)-1)(w(Q)-1)+2,& h(P)=w(Q)=2\mbox{, or}\\
& h(P)\ge 3\mbox{ and } w(Q)\ge 2\\

\end{array}
\right.$
\end{description}
\end{fact}

\begin{proof}
These follow from Proposition~\ref{CMCN}, Proposition~\ref{AMAN}, and Theorem~\ref{CMAN}.
\end{proof}

A lower bound of $RR(P,Q)$ in terms of the height and the width of the posets can be obtained analogously if one can determine $RR(A_m,C_n)$.
Although we only determine $RR(A_2, C_n)$, we present an upper bound for $RR(A_m,C_n)$ here.

\begin{proposition}\label{AMCN}
There exists some constant $C$ such that
$ RR(A_m,C_n)\le (m-1)n+C$.
\end{proposition}

\begin{proof}
We prove by induction on $n$, using an idea similar to Theorem~\ref{A2CN}.
Suppose we already have $RR(A_m, C_{n^*})=k$ for some $n^*$.
Consider any color $c$ on  $\B_{k+m-1}$.
Assume that there is no monochromatic $A_m$ under $c$.
Then we will show there is rainbow $C_{n^*+1}$.
Let $c([k+m-1])=1$ and $\F_1$ the family of nonempty subsets of $[k+m-1]$ colored by 1.
Since the maximum size of an antichain in $\F_1$ is at most $m-1$, by
the well-known Dilworth's theorem~\cite{DIL}, we can partition $\F_1$ into at most $m-1$ disjoint chains, say $\C_1,\C_2,\ldots,\C_\ell$ and $\ell\le m-1$.
For each $\C_i$, let $x_i$ be a common element in each subset in $\C_i$.
Let $S=[k+m-1]\setminus\{x_1,\ldots,x_\ell\}$. Observe that $|S|\ge k$ and the  color of each subset of $S$, except for $\varnothing$, is not $1$.
As before, we consider a new coloring $c'$ on the subsets of $S$ with $c'(X)=c(X)$ for all nonempty $X\subseteq S$ and $c'(\varnothing)=c(S)$.
Because $|S|\ge k$, we can find a rainbow chain $\C$ of subsets of $S$ isomorphic to $C_{n^*}$. If $\C$ does not contain $\varnothing$, then the subsets in $\C\cup \{[k+m-1]\}$ form a rainbow $C_{n^*+1}$ under the coloring $c$.
If $\C$ contains $\varnothing$, then let $\C'=\C\setminus \{\varnothing\}\cup\{S\}$. The subsets in $\C'$ together with $[k+m-1]$ form a rainbow $C_{n*+1}$ under $c$.
Conclusively, for $n\ge n^*$, $RR(A_m,C_{n+1})-RR(A_m,C_{n})\le m-1$  and $RR(A_m,C_n)\le n(m-1)+k-n^*(m-1)$.
\end{proof}

An upper bound for general $P$ and $Q$ can be deduced from the definition of the 2-dimension.
\begin{fact}\label{UPB}
For any poset $P$ and $Q$,
we have
\[RR(P,Q)\le RR(B_{\dim_2(P)},B_{\dim_2(Q)}).\]
\end{fact}
These bounds are quite loose, but our results in Section 2 show that the inequalities in Fact~\ref{LWB} and~\ref{UPB} are tight for some posets.

The inequality in Fact~\ref{UPB} stimulates us to concentrate on evaluating the Boolean rainbow Ramsey number of the Boolean posets.
In the last paragraph, we see that the Boolean Ramsey number $R_{2^n-1}(B_m)$ is a lower bound for $RR(B_m,B_n)$.
Here we give an upper  bound for $RR(B_m,B_n)$ also using the Boolean Ramsey numbers.

\begin{theorem}\label{BMBN}
For the Boolean posets $B_m$ and $B_n$, we have
\[RR(B_m,B_n)\le \sum_{i=1}^{2^n-1}R_i(B_m)\]
\end{theorem}

\begin{proof}
We consider any coloring $c$ on $\B_N$, where $N=\sum_{i=1}^{2^n-1}R_i(B_m)$.
Assume that $\B_N$ does not contain a monochromatic $B_m$ under $c$.
Then we prove there exists a rainbow $B_n$.

Let us arrange the nonempty subsets of $[n]$ as $I_1,\ldots, I_{2^n-1}$
satisfying the condition that $|I_i|\le |I_j|$ if $i<j$.
Partition $[N]$ into $2^n-1$ disjoint subsets $X_i$'s such that
$|X_i|=R_i(B_m)$. For $X,Y\subseteq [N]$, denote $[X,Y]=\{Z\mid X\subseteq Z\subseteq Y\}$.
Note that $[X,Y]$ is isomorphic to $\B_{|Y-X|}$.
Define $Y_0=\varnothing$. In $[Y_0,X_1]$, there exists some set whose color is distinct to $c(Y_0)$,
otherwise $[Y_0,X_1]$ form a monochromatic $B_m$ under $c$. We pick such a set and call it $Y_1$.
Similarly, we can find a set in $[Y_0,X_2]$ whose color is distinct to $c(Y_0)$ and $c(Y_1)$,
otherwise the sets in $[Y_0,X_2]$, isomorphic to $B_{R_2(B_m)}$,
are colored by only two colors, which contains a monochromatic $B_m$.
Again, we pick such a set and call it $Y_2$.
By the same reasoning, we can find disjoint $Y_0,\ldots,Y_n$ with all distinct colors.
These sets will play the roles of the empty set and the singletons of the rainbow $B_n$.

For $n+1\le i \le n+\binom{n}{2}$, we have $|I_i|=2$.
We will find $Y_i$ accordingly so that each $Y_i$ represents a 2-subset of the desired $B_n$.
Suppose $I_i=\{j,k\}$. Then $[Y_j\cup Y_k, X_i\cup Y_j\cup Y_k]$ is isomorphic to $B_{R_i(B_m)}$.
As before, we can find a set with color distinct to $c(Y_0),\ldots ,c(Y_{i-1})$.
Therefore, the set $Y_i$ can be chosen from $[\cup_{j\in I_i}Y_j, X_i\cup(\cup_{j\in I_i}Y_j)]$,
and the resulting sets $Y_0,\ldots, Y_{2^n-1}$ form a rainbow $B_n$.
\end{proof}

{\bf Remark.}
By Theorem~\ref{AW} (Theorem 5, \cite{AW}), $R_k(P)\le C\cdot k$, where $C$ is a constant determined by $P$ only.
Thus, Theorem~\ref{BMBN} is a proof of the existence of $RR(P,Q)$ without using Theorem~\ref{JLM}.
We will have more discussions on this aspect in next section.\bigskip

The rest of the section will be the proof of  Theorem~\ref{B2B2} (3), $RR(B_2,B_2)=6$.
Let us first determine the Boolean Ramsey number $R_3(B_2)$ which will be used as a lower bound for $RR(B_2,B_2)$
in the proof of Theorem~\ref{B2B2} (3).

\begin{theorem}\label{R3B2}
The Boolean Ramsey number $R_3(B_2)$ is equal to 6.
\end{theorem}

\begin{proof}
For any $k$, it is easy to see $R_k(B_2)\ge 2k$ by coloring every consecutive two levels of $\B_{2k-1}$ with one color.
Thus, the remaining of the proof is to show the $R_3(B_2)\le 6$.

Suppose there exists a 3-coloring $c$ on $\B_6$ without any monochromatic $B_2$ under $c$.
By Theorem 1 in~\cite{AW} that $R_2(B_2)=4$, if a collection $\F$ of sets forms a copy of $B_4$, then $\F$ contains subsets of all the three colors.

Given $\B_n$, for distinct $i$ and $j$ in $[n]$ define the family $\B_{i,j}:=\{X\subset [n] \mid i\in X,j\not\in X\}$.
Note that $\B_{i,j}$ is isomorphic to $\B_{n-2}$, and $\B_{i,j}$ and $\B_{j,i}$ are disjoint.
Now for the coloring $c$, if $c([6])=c(\varnothing)$,
then we consider $\B_{i,j}$ and $\B_{j,i}$ for distinct $i,j\in [6]$.
Since $\B_{i,j}$ and $\B_{j,i}$ are isomorphic to $B_4$ and $R_2(B_2)=4$,
there exist $X\in\B_{i,j}$ and $Y\in\B_{j,i}$ with $c(X)=c(Y)=c([6])=c(\varnothing)$
which form a monochromatic $B_2$, a contradiction.
Thus, we may assume that $c(\varnothing)=1$ and $c([6])=2$.

\textbf{Claim 1:} At least five sets in $\binom{[6]}{ 1}$ and at least five sets in $\binom{[6]}{5}$ are  of color $3$.

We will only show that there is at most one set of color 1 and no set of color 2 in $\binom{[6]}{1}$.
The same idea applies for showing at least five sets of color $3$ in $\binom{[6]}{5}$.
Suppose to the contrary that $c(\{i\})=c(\{j\})=1$ for distinct $i,j\in[6]$.
Note that the sets in $\F=\{X\mid \{i,j\}\subseteq X \subseteq [6]\}$ form a $B_4$.
Thus, there exists $X_0\in\F$ such that $c(X_0)=1$.
Then the four sets $\varnothing,\{i\},\{j\},$ and $X_0$ form a monochromatic $B_2$, a contradiction.
If $c(\{i\})=2$ for some $i\in[6]$,
then the subsets of color 2 in the family $\G'=\{X\mid \{i\}\subseteq X\subseteq  [6]\}$ must form a chain $\C$.
Suppose that $\C$ is contained in $\{\{i\},\{i,j\},\ldots,[6]\}$ for some $j$.
Then $\G=(\B_{i,j}-\{\{i\}\})\cup\{\varnothing\}$ is isomorphic to $\B_4$ and contains no subset of color 2.
As before, $\G$ contains a monochromatic $B_4$. Thus, $c(\{i\})\neq 2$ for any $i$.
Hence at least five sets in $\binom{[6]}{ 1}$  are of color $3$. The claim is proved.

Define $\B_{i,j}^{\updownarrow}=\{X\mid\{i\}\subsetneq X\subsetneq[6]\setminus\{j\}\}\cup\{\varnothing,[6]\}$ for $i
\neq j\in[6]$. This is also a collection of sets forming $B_4$.
As before, there exists $X_{i,j}\in \B_{i,j}^{\updownarrow}$ with $c(X_{i,j})=3$. Since $c(\varnothing)=1$ and $c([6])=2$,
we have $X_{i,j}\in\binom{[6]}{2}\cup\binom{[6]}{3}\cup\binom{[6]}{4}$.
However, we will prove that this is impossible.

\textbf{Claim 2:} For any set $X\in \binom{[6]}{2}\cup\binom{[6]}{3}\cup\binom{[6]}{4}$, either $c(X)=1$ or $c(X)=2$.

By Claim 1, we can pick a pair of integers $i_0$ and $j_0$ such that $c(X)=3$ for any
$X\in (\binom{[6]}{1} \cup \binom{[6]}{5}) \setminus\{\{i_0\},[6]\setminus\{j_0\}\}$.

\textbf{Case 1:} $i_0=j_0$.

For any $k\in[6]\setminus\{i_0\}$, we have $c(\{i_0,k\})\neq 3$; otherwise we pick $\ell\in [6]\setminus(\{i_0,k\}\cup X_{k,i_0})$
such that $\{k\}$, $\{i_0,k\}$, $X_{k,i_0}$ and $[6]\setminus\{\ell\}$ form a monochromatic $B_2$ of color $3$.
Similarly, $c([6]\setminus\{i_0,k\})\neq 3$ for any $k\in[6]\setminus\{i_0\}$, otherwise $\{\ell\}$, $X_{\ell ,i_0}$, $[6]\setminus\{i_0,k\}$ and $[6]\setminus\{k\}$
form a monochromatic $B_2$ of color $3$.

For distinct $h,k\in[6]\setminus\{i_0\}$ we also have $c(\{h,k\})\neq 3$; otherwise we pick $\ell\in [6]\setminus(\{i_0,h,k\}\cup X_{k,h})$ such that
$\{k\}$, $\{h,k\}$, $X_{k,h}$ and $[6]\setminus\{\ell\}$ form a monochromatic $B_2$ of color 3.
We cannot find such an $\ell$ only when $\{i_0,h,k\}\cup X_{k,h}=[6]$, which means $X_{k,h}=[6]\setminus\{i_0,h\}$.
This is impossible since $c([6]\setminus\{i_0,h\})\neq 3$.
Similarly, $c([6]\setminus\{h,k\})\neq 3$ for distinct $h,k\in[6]\setminus\{i_0\}$, otherwise we pick $\ell\in [6]\setminus(\{i_0,h,k\}\cup X_{k,h})$ such that
$\{\ell\}$, $X_{k,h}$, $[6]\setminus\{h,k\}$ and $[6]\setminus\{h\}$ form a monochromatic $B_2$ of color 3.
We cannot find such an $\ell$ only when $\{i_0,h,k\}\cup X_{k,h}=[6]$, which means $X_{k,h}=[6]\setminus\{i_0,h\}$.
This is impossible since $c([6]\setminus\{i_0,h\})\neq 3$.
Now the only possible size of $X_{i,j}$ is 3.

For distinct $h,k\in[6]\setminus\{i_0\}$, since $|X_{h,k}|=3$, we can pick $\ell\in [6]\setminus(\{i_0,h,k\}\cup X_{h,k})$.
Note that $c(\{i_0,h,k\})\neq 3$, otherwise $\{k\}$, $\{i_0,h,k\}$, $X_{h,k}$ and $[6]\setminus\{\ell\}$ form a monochromatic $B_2$ of color 3.
Similarly, we also have $c([6]\setminus\{i_0,h,k\})\neq 3$ for distinct $h,k\in[6]\setminus\{i_0\}$, otherwise $\{\ell\}$, $[6]\setminus\{i_0,h,k\}$, $X_{h,k}$ and $[6]\setminus\{k\}$ form a monochromatic $B_2$ for some $\ell\in ([6]\setminus\{i_0,h,k\})\cap X_{h,k}$.
Therefore, $c(X)\neq 3$ for all $X\in\binom{[6]}{3}$.
As a consequence, $X_{i,j}$ does not exist for any distinct $i,j\in[6]$.

\textbf{Case 2:} $i_0\neq j_0$.

First, we have $c(\{i_0,j_0\})\neq 3$, otherwise $\{j_0\}$, $\{i_0,j_0\}$,  $X_{j_0,i_0}$ and $[6]\setminus\{\ell\}$ form a monochromatic $B_2$ of color $3$
with some $\ell\in [6]\setminus(\{i_0,j_0\}\cup X_{j_0,i_0})$.
Similarly, $c([6]\setminus\{i_0,j_0\})\neq 3$, otherwise $\{\ell\}$, $[6]\setminus\{i_0,j_0\}$, $X_{j_0,i_0}$ and $[6]\setminus\{i_0\}$ form a monochromatic $B_2$ of color $3$ with some $\ell\in ([6]\setminus\{i_0,j_0\})\cap X_{j_0,i_0}$.

Moreover, we have $c(\{i_0,k\})\neq 3$ for any $k\in[6]\setminus\{i_0,j_0\},$ otherwise we can find some $\ell\in [6]\setminus (\{i_0,j_0,k\}\cup X_{k,i_0})$ such that $\{k\}$, $\{i_0,k\}$, $X_{k,i_0}$ and $[6]\setminus\{\ell\}$ form a monochromatic $B_2$ of color 3.
We cannot find such an $\ell$ only when $\{i_0,j_0,k\}\cup X_{k,i_0}=[6]$, which means $X_{k,i_0}=[6]\setminus\{i_0,j_0\}$. This is impossible since $c([6]\setminus\{i_0,j_0\})\neq 3$.
Similarly, we have $c([6]\setminus\{k,j_0\})\neq 3$ for any $k\in[6]\setminus\{i_0,j_0\}$,
otherwise $\{\ell \}$, $X_{j_0,k}$, $[6]\setminus\{k,j_0\}$ and $[6]\setminus \{k\}$ form a monochromatic $B_2$ of color 3
with some $\ell \in ([6]\setminus\{i_0,j_0,k\})\cap X_{j_0,k}$.
Again, such an $\ell $ must exist
since $X_{j_0,k}\neq\{i_0,j_0\}$ by the assumption of their colors.

Next, $c(\{k,j_0\})\neq 3$ for any $k\in[6]\setminus\{i_0,j_0\}$, otherwise $\{j_0\}$, $\{k,j_0\}$, $X_{j_0,k}$ and $[6]\setminus\{\ell\}$ form a monochromatic $B_2$ of color $3$ with some $\ell\in [6]\setminus(\{k,i_0,j_0\}\cup X_{j_0,k})$. As before,
$\ell$ does not exist only if $X_{j_0,k}=[6]\setminus \{i_0,j_0\}$.
Similarly, $c([6]\setminus\{i_0,k\})\neq 3$ for any $k\in[6]\setminus\{i_0,j_0\}$.

Finally, $c(\{h,k\})\neq 3$ for all distinct $h,k\in[6]\setminus\{i_0,j_0\}$,
otherwise $\{k\}$, $\{h,k\}$, $X_{k,h}$ and $[6]\setminus\{\ell\}$ form a monochromatic $B_2$ with some $\ell\in [6]\setminus(\{j_0,h,k\}\cup X_{k,h})$.
Note that $X_{k,h}\neq [6]\setminus\{j_0,h\}$ since $c([6]\setminus\{j_0,h\})\neq 3$ has been shown above. Similarly, $c([6]\setminus\{h,k\})\neq 3$ for all distinct $h,k\in[6]\setminus\{i_0,j_0\}$ is also true. Now we have $|X_{i,j}|=3$.

Assume $X_{i_0,j_0}=\{i_0,h,k\}$ for some $h,k\in[6]\setminus\{i_0,j_0\}$. Then this leads to $X_{h,k}=[6]\setminus\{i_0,j_0,k\}$,
otherwise $\{h\}$, $X_{i_0,j_0}$, $X_{h,k}$ and $[6]\setminus\{\ell\}$ form a monochromatic $B_2$
with some $\ell\in [6]\setminus(X_{i_0,j_0}\cup X_{h,k}\cup\{j_0\}).$ With the same argument, one has $X_{k,h}=[6]\setminus\{i_0,j_0,h\}$.
However, we now have four sets $\{\ell\}$, $X_{h,k}$, $X_{k,h}$ and $[6]\setminus\{i_0\}$ forming a monochromatic $B_2$ with $\ell\in [6]\setminus\{i_0,j_0,h,k\}$, and a contradiction occurs.

Conclusively, the 3-coloring $c$ on $\B_6$ without a monochromatic $B_2$ cannot exist.
\end{proof}

To complete the proof of Theorem~\ref{B2B2} (3), we need more tools.
Let $\vee$ and $\wedge$ denote the posets consisting of three elements $a,b$, and $c$ with $a\le b, c$ and  $a,b\le c$, respectively.

\begin{lemma}\label{lemv^}
For any poset $P$, we have
\[R_2(P)\leq RR(P,\{\vee,\wedge\})\leq \max\{R_2(P),\dim_2(P)+2\}.\]
Moreover, if $2h(P)\geq \dim_2(P)+4$, then $RR(P,\{\vee,\wedge\})=R_2(P)$.
\end{lemma}

\begin{proof}
Let $M=RR(P,\{\vee,\wedge\})$ and $N=\max\{R_2(P),\dim_2(P)+2\}$.
For any 2-coloring on $\B_M$, it must contain a monochromatic $P$ since there is no rainbow $\vee$ or $\wedge$ under a 2-coloring. So $M\geq R_2(P)$.

Assume that $c$ is a $k$-coloring on $\B_N$. 
If $k\leq 2$, then $\B_N$ contains a monochromatic $P$ since $N\geq R_2(P)$. So we consider $k\geq 3$.
Without lose of generality, let $c(\varnothing)=1$.
Pick two subsets $X$ and $Y$ of $[n]$ such that  $c(X)=2$ and $c(Y)=3$,
if $X\not\subset Y$ and $Y\not\subset X$, then $\varnothing$, $X$, and $Y$ form a rainbow $\vee$.
So we may assume $X\subset Y$.
If $Y\neq [n]$, then we can pick $i,j\in[n]$ so that $X,Y\in \B_{i,j}$.
Recall that $\B_{i,j}$ and $\B_{j,i}$ are disjoint. We must have $c(X)= 1$ for any $W\in\B_{j,i}$ , otherwise
$\varnothing$, $W$, and one of $X$ or $Y$ form a rainbow $\vee$. Then $\B_{j,i}$ form a monochromatic $B_{N-2}$,
as well as a monochromatic $P$ since $N-2\geq \dim_2(P)$.
If $Y=[n]$, then we pick $i,j\in[n]$ so that $X\in \B_{i,j}$. Now we must have $c(Z)=2$ for every $Z\in \B_{j,i}$,
otherwise either $\varnothing$, $X$, and $Z$ form a rainbow $\vee$, or $X$, $Z$, and $Y$ form a rainbow $\wedge$.
As before, the monochromatic $\B_{j,i}$ contains $P$. Therefore, $N\ge RR(P,\{\vee,\wedge\})$.

To see the ``moreover'' part, we prove the fact $R_2(P)\geq 2h(P)-2$
by constructing a 2-coloring on $\B_{2h(P)-3}$ without a monochromatic $P$.
Namely, assign two colors to the elements in the bottom $h(P)-1$ levels
and the elements in the top $h(P)-1$ levels of $\B_{2h(P)-3}$, respectively.
Thus, if $2h(P)\geq \dim_2(P)+4$, then $R_2(P)\geq \dim_2(P)+2$.
We have $\max\{R_2(P),\dim_2(P)+2\}=R_2(P)$, and $RR(P,\{\vee,\wedge\})=R_2(P)$.
\end{proof}

{\bf Remark.} Since $\dim_2(B_n)=n$ and $h(B_n)=n+1$, we have $2h(B_n)\ge \dim_2(B_n)+4$ for $n\ge 2$.
So $RR(B_n,\{\vee,\wedge\})=R_2(B_n)$ for $n\ge 2$.

\bigskip

\noindent{\em Proof of Theorem~\ref{B2B2} (3).}
Consider any $k$-coloring $c$ on $\B_6$.
If $k\leq 3$, then $\B_6$ contains a monochromatic $B_2$ since $R_3(B_2)=6$ by Theorem~\ref{R3B2}.
We assume $k\ge 4$.

Suppose that $c(\varnothing)=c([6])$.
Without lose of generality, let $c(\varnothing)=c([6])=1$ .
For $i\neq j$,
if both the families $\B_{i,j}$ and $\B_{j,i}$ contain some sets of color 1,
say $c(X_{i,j})=c(X_{j,i})=1$ for $X_{i,j}\in \B_{i,j}$ and $X_{j,i}\in \B_{j,i}$,
then $\varnothing$, $X_{i,j}$, $X_{j,i}$, and $[6]$ form a monochromatic $B_2$.
We may suppose $\B_{i,j}$ contains no set of color 1.
Note that $\B_{i,j}$ is isomorphic to $\B_4$ and $RR(B_2,\{\vee,\wedge\})=4$.
Either $\B_{i,j}$ contains a monochromatic $B_2$, or a rainbow $\vee$, or a rainbow $\wedge$.
For the latter two cases, the rainbow $\vee$ (or $\wedge$) together with $[6]$ (or $\varnothing$) form a rainbow $B_2$.

If $c(\varnothing)\ne c([6])$, let $c(\varnothing)=1$ and $c([6])=2$.
Pick two subsets $X,Y\subset [6]$ with $c(X)=3$, $c(Y)=4$. If both $X\not\subset Y$ and $Y\not\subset X$, then
the four sets $\varnothing$, $X$, $Y$, and $[6]$ form a rainbow $B_2$.
So we may assume $X\subset Y$ or $Y\subset X$.
Moreover, we have $X,Y\in\B_{i,j}$ for some $i,j\in [6]$.
If $c(Z)\geq 3$ for some $Z\in\B_{j,i}$, then there is a rainbow $B_2$ formed by $\varnothing$, $Z$, $[6]$, and one of $X$ and $Y$.
If $c(Z)=1$ or 2 for all $Z\in\B_{j,i}$,
then there exists a monochromatic $B_2$ since $\B_{j,i}$ is isomorphic to $\B_{4}$ and $R_2(B_2)=4$.

Therefore, we have $RR(B_2,B_2)\le 6$, and the proof is completed.
\qed

\section{Discussions}

We gave an upper bound for $R(A_m, C_n)$ in Proposition~\ref{AMCN}.
A lower bound can be deduced in the following way:
Recall that $N_{m,n}$ is the minimum integer $N$ such that $\binom{N}{\lfloor N/2\rfloor}\ge (m-1)(n-1)+1$.
Now consider the set $S$ that is a union of disjoint sets
$X=\{x_1,\ldots, x_{N_{m,2}-1}\}$ and $ Y=\{y_1,\ldots, y_{n-2}\}$.
For each subset $Z\subseteq S$, color it according to $Z\cap Y$.
That is, we assign the same color to $Z_1$ and $Z_2$ if and only if if $Z_1\cap Y=Z_2\cap Y$.
Thus, the subsets in a color class form a Boolean poset $B_{N_{m,2}-1}$, which do not contain $A_m$.
On the other hand, if $Z_1,\ldots, Z_k$ form a rainbow chain,
then $Z_1\cap Y,\ldots, Z_k\cap Y$ also form aa rainbow chain.
However, any rainbow chain formed by subsets of $Y$ has length at most $n-1$.
So there does not exist a rainbow $C_n$.
Therefore, we have $RR(A_m,C_n)\ge N_{m,2}+n-3$.
Obviously, there is a huge gap between the lower and upper bounds.
It would be interesting if one can tighten the gap.

One of our main interests is to estimate $RR(B_m,B_n)$.
Theorem~\ref{JLM} provides us an upper bound $RR(B_m,B_n)\le n2^{(2n+1)2^{m-1}-2}$.
We also have $RR(B_m,B_n)\le \sum_{i=1}^{2^n-1}R_i(B_m)$.
In the following, we give an estimation of $\sum_{i=1}^{2^n-1}R_i(B_m)$ in terms of $m$ and $n$.
Recall that in the Remark after Theorem~\ref{BMBN}, we mentioned that $R_k(P)\le C\cdot k$,
where $C$ is a constant determined by $P$.
Indeed, this is consequence from a result of M\'{e}roueh in~\cite{M}.
Let us introduce the Lubell functions and Forbidden subposet problems studied intensively recently.
The {\em Lubell function} $\bar{h}_n(\F)$ for a family $\F$ of subsets of $[n]$ is defined as
\[
\bar{h}_n(\F):=\sum_{F\in \F}\frac{1}{\binom{n}{|F|}}.
\]
It is obvious that $\bar{h}_n(\B_n)=n+1$.
For the background of the Lubell functions and Forbidden subposet problems,
we refer the readers to a recent survey~\cite{GL}.
In the literature, Lu and Milans~\cite{LM} proposed the conjecture:

\begin{conjecture}{\rm\cite{LM}}\label{LM}
Let $\lambda_n(P)$ be the maximum value of the Lubell function
for families of subsets of $[n]$ not containing $P$ as a subposet. Then $\limsup \lambda_n(P)$ is finite.
\end{conjecture}

M\'{e}roueh~\cite{M} verified this conjecture by showing that if  $\bar{h}_n(\F)> 1000m^7 16^m$,
then the family $\F$ must contain $B_m$,
as well as any poset whose 2-dimension is at most $m$.
Thus, when $N\ge  1000m^7 16^m k$, any $k$-coloring on $\B_N$ must contain one color class $\F_i$ of subsets of color $i$ with
$\bar{h}(\F_i)>1000m^7 16^m$, and hence a monochromatic $B_m$ of color $i$.
In other words, $R_k(B_m)\le 1000m^7 16^m k$.
Using his result, we conclude that
\[RR(B_m,B_n)\le \sum_{i=1}^{2^n-1}R_i(B_m)\le 2^{n-1}(2^n-1)1000m^7 16^m< m^7 2^{2n+4m+9}. \]
This upper is better than that derived from Theorem~\ref{JLM} under certain circumstances.

On the other hand, we have the lower bound $RR(B_m, B_n)\ge R_{2^n-1}(B_m)$.
However, $R_k(B_m)$ is unknown in general.
Generalizing the construction in the proof of Theorem~\ref{R3B2},
showing that $R_k(B_2)\ge 2k$, shows $R_k(B_m)\ge mk$.
When $m=2$, we know $R_k(B_2)=2k$ for $k\le 3$.
By the well-known fact $\lambda_n(B_2)\le 2\frac{2}{3}$ in the literature of forbidden subposet problems~\cite{GLL,GMT,KMY,LM},
one can show  $R_k(B_2)\le (2\frac{2}{3})k$, and hence $8\le R_4(B_2)\le 10$.
An observation is that if there exists a 4-coloring $c$ on $\B_8$ without a monochromatic $B_2$,
then $c(\varnothing)\neq c([8])$, otherwise we can find a monochromatic $B_2$ formed by
$\varnothing$, $[8]$, some $X_{i,j}\in \B_{i,j}$ and some $X_{j,i}\in \B_{j,i}$ inductively.
However, the remaining case is a little complicated, and we expect a more clever argument.
So we leave the following open problem:
\begin{conjecture}
For all $k\ge 1$, the Boolean Ramsey number $R_k(B_2)=(2+o_k(1))k$.
\end{conjecture}

\begin{center}
{\bf Acknowledgement}
\end{center}
The research is also partially supported by Taiwanese-Hungarian Mobility Program of the Hungarian Academy of Sciences and by Ministry of Science and Technology Project-based Personnel
Exchange Program.

\end{document}